\documentclass[11pt,a4paper]{article}

\usepackage{authblk}

\usepackage{mathrsfs}
\usepackage{amsfonts}
\usepackage{cite}
\usepackage{amsmath,amssymb,amscd,amsthm}
\usepackage{graphics}
\usepackage{enumerate}
\input xy
\xyoption{all}

\numberwithin{equation}{section}

\usepackage{scalefnt}   

\usepackage{comment}
\usepackage{setspace}

\usepackage[paperwidth=210mm,
            paperheight=297mm,
            a4paper,
            twoside,
            top=0.8in, left=0.9in, right=0.9in, bottom=1in]{geometry}

\usepackage[active]{srcltx}

\theoremstyle{definition}
\newtheorem{lemma}{Lemma}[section]
\newtheorem*{claim}{Claim}
\newtheorem{prop}[lemma]{Proposition}

\newtheorem{thm}[lemma]{Theorem}

\newcommand{\C}{\mathbb{C}}
\newcommand{\N}{\mathbb{N}}
\newcommand{\Z}{\mathbb{Z}}

\def\a{\alpha}

\def\O{\Omega}

\usepackage{float}
\restylefloat{figure}

\allowdisplaybreaks

\begin{document}
\title{New simple modules for the $W$-algebra $W(2,2)$}

\author{Hongjia Chen}
\author{Dashu Xu\footnote{Corresponding author: dox@mail.ustc.edu.cn}}
\affil{School of Mathematical Sciences, University of Science and Technology of China, Hefei 230026, Anhui, P. R. China}

\date{}
\maketitle
\begin{abstract}
In this paper, we construct a novel class of simple modules for the $W$-algebra $W(2,2)$. 
Our approach involves taking tensor products of finitely many non-weight simple modules $\Omega(\lambda,\a,h)$ 
with an arbitrary simple restricted module. 
We provide a necessary and sufficient condition for these modules to be simple, and subsequently determine their isomorphism classes.
Through a comparative analysis with other known simple modules in the literature, we establish that these constructed modules are generically new.


\end{abstract}

\section{Introduction}
The $W$-algebra $W(2,2)$ is an infinite dimensional Lie algebra arising naturally in the context of vertex operator algebras\cite{Zhang&Dong}.
It also has certain physical interpretations \cite{MU}.
Denote this algebra by $\mathcal{W}$. 
It has a basis $\left\{ L_n, W_n, C\mid n\in\Z\right\}$ and satisfies the non-trivial relations
\begin{align*}
[L_m,L_n]&=(n-m)L_{m+n}+\delta_{m+n,0}\frac{m^3-m}{12}C,\\
[L_m,W_n]&=(n-m)W_{m+n}+\delta_{m+n,0}\frac{m^3-m}{12}C.
\end{align*}
The subalgebra generated by $\left\{L_n\mid n\in\Z\right\}$ is the so-called Virasoro Lie algebra.
The algebra $\mathcal{W}$ is an extension of the Virasoro algebra, whereas, its representation theory is rather different from that of the Virasoro algebra. 
For example, it was shown in \cite{Radobolja} that subsingular vectors may exist in Verma modules of $\mathcal{W}$.

Extensive studies have been conducted on the Lie algebra $\mathcal{W}$ from various perspectives. 
In terms of structural theory, various operators on the algebra $\mathcal{W}$, such as the derivations, 
local derivations, 2-local derivations and automorphism groups of $\mathcal{W}$ have been explicitly determined in \cite{Gao&Jiang&Pei,Wu&Gao&Liu,Tang}.
In addition, the pre-Lie algebra structures on $\mathcal{W}$ subject to some compatible conditions are completely classified in \cite{Chen&Li}.
Regarding the representation theory, Zhang and Dong defined the Verma module of $\mathcal{W}$ in \cite{Zhang&Dong}. 
They gave a necessary and sufficient condition for a Verma module to be simple, initiating the study on weight modules of $\mathcal{W}$. 
Following this, Radobolja delved into the structures of the Verma modules and investigated the tensor products of weight modules of $\mathcal{W}$, 
yielding weight $\mathcal{W}$-modules with infinite dimensional weight spaces \cite{Radobolja}.
Based on the $A$-cover theory developed by Billing and Furtony in \cite{B&F}, Cai, L\"{u} and Wang classified simple Harish-Chandra modules of a Lie algebra related to the Virasoro algebra, 
which includes the algebra $\mathcal{W}$ as an example.
Motivated by the structural features of $\mathcal{W}$, Wang systematically studied the Whittaker modules of a class of graded Lie algebras with $\mathcal{W}$ being a special case.
Adamović and Radobolja found interesting applications of vertex algebra theory in the study of representations of $\mathcal{W}$. 
They proved the $W(2,2)$ vertex operator algebra can be embedded into the twisted Heisenberg-Virasoro vertex operator algebra \cite{A&R}.
Subsequently, they gave a complete description of the structures of the simple highest weight modules of the twisted Heisenberg-Virasoro algebra as $\mathcal{W}$-modules \cite{AR}.
In the vein of the work by Mazorchuk and Zhao in \cite{Ma&Zhao}, Chen, Hong and Su constructed a family of simple restricted modules over $\mathcal{W}$, which include the simple highest weight modules and Whittaker modules. They also characterized the properties of these constructed modules conceptually \cite{Chen&Hong&Su}.
The study of representations of $\mathcal{W}$ is also conducive to the research on that of the Virasoro algebra.
In \cite{C&G}, Chen and Guo constructed a new family of Virasoro modules arising from the $U(\C L_0\oplus\C W_0)$-free modules of the algebra $\mathcal{W}$ in \cite{Chen&Guo}.
For more research on the representations of the algebra $\mathcal{W}$, we refer the readers to \cite{Chen&Guo,Jiang&Pei&Zhang,Liu&Gao&Zhu,Zhang&Tan,Wang&Li}
and the references therein.

This paper is concerned with the representations of the Lie algebra $\mathcal{W}$.
Our target is to construct new, simple modules for the algebra $\mathcal{W}$. 
Initially, we explore the tensor products of arbitrary finite non-weight simple modules,
as constructed in \cite{Chen&Guo}, with an arbitrary simple restricted module. 
We establish a criterion for such a module to be simple and determine their isomorphism classes. 
Subsequently, we prove that these modules are generically not isomorphic to other known simple modules documented in the literature, 
implying that they are indeed new simple modules of the algebra $\mathcal{W}$.

Throughout this paper, we use $\N$, $\Z$, $\C$ and $\C^*$ to denote the sets of non-negative integers, integers, complex numbers and non-zero complex numbers respectively.
Additionally, for any arbitrary Lie algebra $\mathfrak{g}$, we employ $U(\mathfrak{g})$ to denote its universal enveloping algebra.

\section{Preliminaries}

In this section, we recall some facts about the algebra $W(2,2)$. 

The Lie algebra $\mathcal{W}$ is a $\Z$-graded algebra with a decomposition
\[
\mathcal{W}=\bigoplus_{n\in\Z}\mathcal{W}_n,\quad
[\mathcal{W}_n,\mathcal{W}_m]\subset \mathcal{W}_{n+m},
\]
where for each $n\in\Z$, $\mathcal{W}_n=\C L_n\bigoplus\C W_n\bigoplus\delta_{n0}\C C$.

A $\mathcal{W}$-module $V$ is called a \textit{restricted module} if for every $v\in V$, there exists a sufficiently large integer $i_v \in\N$ such that $\mathcal{W}_n v=0$ for all $n\ge i_v$.
Note that there are some explicitly constructed simple restricted $\mathcal{W}$-modules in \cite{Zhang&Dong,Wang&Li,Chen&Hong&Su}.

For $\lambda\in\C^*$, $\a\in\C$ and $h(t)\in\C[t]$. Denote by $\O(\lambda,\a)=\C[s]$ and $\O(\lambda,\a,h)=\C[s,t]$, respectively. Define the actions of $\mathcal{W}$ on $\O(\lambda,\a)$ as
\begin{align*}
L_n(s^p)&=\lambda^n(s-n\a)(s-n)^p, \quad  W_n(s^p)=C(s^p)=0,
\end{align*}
where $p\in\N$. 
Then by Theorem 3 in \cite{Chen&Guo}, we know that $\O(\lambda,\a)$ is a $\mathcal{W}$-module. 
Moreover, $\O(\lambda,\a)$ is simple if and only if $\a\in\C^*$.
Set $g(t)=\dfrac{h(t)-h(\a)}{t-\a}$ and define linear maps
\begin{equation}\label{F}
\begin{split}
F:\C[t]&\to\C[t]\\
f(t)&\mapsto g(t)f(t)-f'(t)
\end{split}
\end{equation}
and
\begin{equation}\label{G}
\begin{split}
G:\C[t]&\to\C[t]\\
f(t)&\mapsto tF\big(f(t)\big)+h(\a)f(t)
\end{split}
\end{equation}
where $f'(t)$ is the derivative of $f(t)$.
Now, we define the actions of $\mathcal{W}$ on $\O(\lambda,\a,h)$ as
\begin{align*}
L_n(s^pt^q)&=\lambda^n(s-n)^p\big(st^q+nG(t^q)-n^2\a F(t^q)\big),\\
W_n(s^pt^q)&=\lambda^n(t-n\a)(s-n)^pt^q,\\
C(s^pt^q)&=0,
\end{align*}
where $p,q\in\N$. Then it follows from Proposition 3.1 in \cite{Chen&Guo} that $\O(\lambda,\a,h)$ is a $\mathcal{W}$-module. Furthermore, $\O(\lambda,\a,h)$ is simple if and only if $\a\in\C^*$.

The following proposition determines the homomorphisms (not necessarily isomorphisms) between the modules $\O(\lambda,\a,h)$, constituting a slight generalization of Proposition 3.2 from \cite{Chen&Guo}.
For $f(t)\in\C[t]$, we define
\[
\chi(f)=\left\{\begin{matrix}
	1&f\in\mathbb{N}, \\
	0&f\notin\mathbb{N}.
\end{matrix}\right.
\]
\begin{prop}
We have
\[
\begin{split}
\mathrm{Hom}_{\mathcal{W}}\left(\O(\lambda_1,\alpha_1,h_1),\O(\lambda_2,\alpha_2,h_2)\right)
=&\;\delta_{\lambda_1\lambda_2}\delta_{\alpha_1\alpha_2}\delta_{h_1h_2}(1-\delta_{\alpha_10})\C\mathbf{id}_{\C[s,t]}\\
&+\delta_{\lambda_1\lambda_2}\delta_{\alpha_1\alpha_2}\delta_{\alpha_10}\chi(h_2-h_1)\C t^{h_2(t)-h_1(t)}.
\end{split}
\]
\end{prop}
\begin{proof}
Suppose $\Phi:\O(\lambda_1,\alpha_1,h_1)\to\O(\lambda_2,\alpha_2,h_2)$ is a non-zero homomorphism and $\Phi(1)=v(s,t)$.
Then the actions of $L_0$ and $W_0$ indicate $\Phi(u(s,t))=u(s,t)v(s,t)$ for every $u(s,t)\in\O(\lambda_1,\alpha_1,h_1)$.
From the equation $\Phi(W_11)=W_1\Phi(1)$, we deduce that $\lambda_1(t-\alpha_1)v(s,t)=\lambda_2(t-\alpha_2)v(s-1,t)$.
Assume that 
\[
v(s,t)=\sum_{k=0}^{n}v_k(t)s^k,
\]
where $v_k(t)\in\C[t]$ and $v_n(t)\ne0$.
Then we have
\[
\lambda_1(t-\alpha_1)v_n(t)=\lambda_2(t-\alpha_2)v_n(t).
\]
It follows that $(\lambda_1,\alpha_1)=(\lambda_2,\alpha_2)$, and hence $v(s,t)=v(t)\in\C[t]$.
Denote by $\alpha:=\alpha_1=\alpha_2$.
For each $m\in\Z$, we have $\Phi(L_m1)=L_m\Phi(1)$, which implies
\begin{equation}\label{fln}
\left(m(h_1(t)-h_2(t))-m(m-1)\alpha(g_1(t)-g_2(t))\right)v(t)=m(m\alpha-t)v'(t),
\end{equation}
where $g_i(t)=\dfrac{h_i(t)-h_i(\alpha)}{t-\alpha}$ for $i=1,2$.
Comparing the coefficients of $m^2$ on both sides, we have
\[
\alpha(g_2(t)-g_1(t))v(t)=\alpha v'(t).
\]
If $\alpha\ne0$, then we have $g_1(t)=g_2(t)$ and $v(t)\in\C^*$.
Hence, it is straightforward to see $h_1(t)=h_2(t)$.
If $\alpha=0$, then equation \eqref{fln} becomes $tv'(t)=(h_2(t)-h_1(t))v(t)$.
It follows that $h_2(t)-h_1(t)=n\in\N$ and $v(t)\in\C^*t^n$.
It remains to prove $\Phi(u(s,t))=u(s,t)t^n$ is a homomorphism. 
This can be verified by direct computations.
\end{proof}
\section{Simple tensor product modules}
Let $m$ be a positive integer. 
For $1\le k\le m$, suppose $\lambda_k,\a_k\in\C^*$ and $h_k(t_k)\in\C[t_k]$. 
Then as a vector space, we have $\O(\lambda_k,\a_k,h_k)=\C[s_k,t_k]$. 
Let $g_k(t_k)=\dfrac{h_k(t_k)-h_k(\alpha_k)}{t_k-\a_k}$. 
Define the linear maps $F_k:\C[t_k]\to\C[t_k]$ and $G_k:\C[t_k]\to\C[t_k]$ as in \eqref{F} and \eqref{G} respectively. 
For a simple restricted $\mathcal{W}$-module $V$, we have the module of tensor product:
\[
\mathbf{T}=\bigotimes_{k=1}^{m}\O(\lambda_k,\a_k,h_k)\otimes V.
\]
For a non-zero element
\begin{equation}\label{f}
g=\sum_{(\textbf{p},\textbf{q})\in E}s_1^{p_1}t_1^{q_1}\otimes \cdots \otimes  s_m^{p_m}t_m^{q_m}\otimes v_{(\textbf{p},\textbf{q})}
\end{equation}
of $\mathbf{T}$, where $(\textbf{p},\textbf{q})=(p_1,\ldots,p_m,q_1,\ldots,q_m)$, $0\ne v_{(\textbf{p},\textbf{q})}\in V$ for all $(\textbf{p},\textbf{q})\in E$ and $E$ is a finite set, we shall simply use $\textbf{s}^{\textbf{p}}\textbf{t}^{\textbf{q}}$ to denote $s_1^{p_1}t_1^{q_1}\otimes \cdots \otimes  s_m^{p_m}t_m^{q_m}$. Then \eqref{f} can be written as
\begin{equation*}
g=\sum_{(\textbf{p},\textbf{q})\in E} \textbf{s}^{\textbf{p}}\textbf{t}^{\textbf{q}}\otimes v_{(\textbf{p},\textbf{q})}.
\end{equation*}
Define $P_k=\mathrm{max}\left \{ p_k \mid (\textbf{p},\textbf{q}) \in E\right \}$ and $E_k=\left\{ (\textbf{p},\textbf{q}) \in E\mid p_k=P_k\right\}$ for $1\le k\le m$.
Recall that for any $n\ge 1$, the lexicographical order on $\Z^n$, a total order, is defined as:
\begin{equation*}
(a_1,\ldots,a_n)\succ(b_1,\ldots,b_n)   \; \Longleftrightarrow  \; 
\exists\, 1\le k\le n \;\mathrm{such}\; \mathrm{that}\; a_1=b_1,\ldots,a_{k-1}=b_{k-1}\; \mathrm{and}\; a_k>b_k.
\end{equation*}
Define the degree of $g$ to be the maximal $ (\textbf{p},\textbf{q})$ in $E$.

A crucial lemma for the arguments of this paper has been proved in \cite{Tan&Zhao}.
\begin{lemma}[Lemma 2 in \cite{Tan&Zhao}]
\label{crucial}
Suppose $\lambda_1,\ldots,\lambda_m\in\C^*$, $s_1,\ldots,s_m\in\Z_{\ge1}$ and $s_1+\cdots+s_m=s$.
For $n\in\Z$, $1\le t\le m$ and $s_1+\cdots+s_{t-1}+1\le k\le s_1+\cdots+s_t$, define $f_k(n)=n^{k-1-\sum_{j=1}^{t-1}s_j}\lambda_t^n$. 
Let $\mathfrak{R}=(y_{pq})$ be the $s\times s$ matrix with $y_{pq}=f_q(p+r-1)$, where $r\ge0$ and $p, q=1,2,\ldots,s$.
Then we have
\begin{equation*}
\det\mathfrak{R}=\prod_{j=1}^{m}(s_j-1)^{!!}\lambda_j^{{s_j(s_j+2r-1)}/{2}} \prod_{1\le i<j\le m}(\lambda_j-\lambda_i)^{s_is_j},
\end{equation*}
where $m^{!!}=m!\times(m-1)!\times\cdots\times 1!$ for $m\in\Z_{\ge 1}$ and $0^{!!}=1$.
\end{lemma}
\subsection{Simplicity}
In this subsection, we shall give a necessary and sufficient condition for the tensor product module to be simple.
First, the following lemmas are needed.
\begin{lemma}
Assume that $\lambda_1,\ldots,\lambda_m$ are pairwise distinct and $N$ is a submodule of $\mathbf{T}$. Then for
\[
0\ne g=\sum_{(\textbf{p},\textbf{q})\in E} \textbf{s}^{\textbf{p}}\textbf{t}^{\textbf{q}}\otimes v_{(\textbf{p},\textbf{q})}\in N
\]
 and $1\le k\le m$, we have
\begin{equation}\label{n}
\sum_{(\textbf{p},\textbf{q})\in E_k} \textbf{s}^{\textbf{p}-P_k e_k}\textbf{t}^{\textbf{q}}\otimes v_{(\textbf{p},\textbf{q})}\in N,
\end{equation}
\begin{equation}\label{nn}
\sum_{(\textbf{p},\textbf{q})\in E} \textbf{s}^{\textbf{p}}\textbf{t}^{\textbf{q}+e_k}\otimes v_{(\textbf{p},\textbf{q})}\in N,
\end{equation}
\begin{equation}\label{nnn}
\sum_{(\textbf{p},\textbf{q})\in E} \textbf{s}^{\textbf{p}+e_k}\textbf{t}^{\textbf{q}}\otimes v_{(\textbf{p},\textbf{q})}\in N,
\end{equation}
where $e_k=(\delta_{1k},\ldots,\delta_{mk})$.
\end{lemma}
\begin{proof}
Choose a sufficiently large integer $i$ such that $\mathcal{W}_nv_{(\textbf{p},\textbf{q})}=0$ for all $n\ge i$ and $(\textbf{p},\textbf{q})\in E$. 
For $n\ge i$, by the action of $W_n$, we have
\begin{align*}
W_ng=
&\sum_{(\textbf{p},\textbf{q})\in E}\sum_{k=1}^{m}s_1^{p_1}t_1^{q_1}\otimes \cdots\otimes W_n(s_k^{p_k}t_k^{q_k}) \otimes\cdots\otimes  s_m^{p_m}t_m^{q_m}\otimes v_{(\textbf{p},\textbf{q})}\\
=&\sum_{(\textbf{p},\textbf{q})\in E}\sum_{k=1}^{m}s_1^{p_1}t_1^{q_1}\otimes \cdots\otimes \lambda_k^{n}(t_k-n\alpha_k)(s_k-n)^{p_k}t_k^{q_k} \otimes \cdots\otimes s_m^{p_m}t_m^{q_m}\otimes v_{(\textbf{p},\textbf{q})}\\
=&\sum_{(\textbf{p},\textbf{q})\in E}\sum_{k=1}^{m}\sum_{j=0}^{p_k+1}(-1)^jn^j\lambda_{k}^{n}s_1^{p_1}t_1^{q_1}\otimes \cdots\otimes \left (\binom{p_k }{ j} t_k
+\alpha_k\binom{p_k }{ j-1}s_k\right ) s_k^{p_k-j}t_k^{q_k} \otimes\cdots\otimes  v_{(\textbf{p},\textbf{q})}.
\end{align*}
Hence, from Lemma \ref{crucial}, we deduce that the elements
\begin{equation}\label{t}
\sum_{(\textbf{p},\textbf{q})\in E}s_1^{p_1}t_1^{q_1}\otimes \cdots\otimes 
\left ( \begin{pmatrix} p_k \\ j \end{pmatrix}s_{k}^{p_k-j}t_k^{q_k+1}+\alpha_k\begin{pmatrix} p_k \\ j-1 \end{pmatrix}s_k^{p_k-j+1}t_k^{q_k} \right )
\otimes \cdots \otimes  s_m^{p_m}t_m^{q_m}\otimes v_{(\textbf{p},\textbf{q})}
\end{equation}
belong to $N$ for $1\le k\le m$ and $0\le j\le P_k+1$. 
Taking $j=P_k+1$ in \eqref{t} and noting that $\a_k\ne0$, we see that \eqref{n} is true. 
Taking $j=0$ in \eqref{t}, we obtain \eqref{nn}. 
For $n\ge i$, the action of $L_n$ implies
\begin{align*}
L_ng
=&\sum_{(\textbf{p},\textbf{q})\in E}\sum_{k=1}^{m}s_1^{p_1}t_1^{q_1}\otimes \cdots\otimes L_n(s_k^{p_k}t_k^{q_k}) 
  \otimes \cdots\otimes s_m^{p_m}t_m^{q_m}\otimes v_{(\textbf{p},\textbf{q})}\\
=&\sum_{(\textbf{p},\textbf{q})\in E}\sum_{k=1}^{m}s_1^{p_1}t_1^{q_1}\otimes \cdots\otimes \lambda_k^{n}(s_k-n)^{p_k}
  \left(s_kt_k^{q_k}+nG_k(t_k^{q_k})-n^2\alpha_kF_k(t_k^{q_k})\right)\\
 &\otimes \cdots\otimes s_m^{p_m}t_m^{q_m}\otimes v_{(\textbf{p},\textbf{q})}\\
=&\sum_{(\textbf{p},\textbf{q})\in E}\sum_{k=1}^{m}\sum_{j=0}^{p_k+2}(-1)^jn^j\lambda_k^n s_1^{p_1}t_1^{q_1}\otimes \cdots \otimes
\left(\begin{pmatrix}  p_k\\j \end{pmatrix} s_k^{p_k-j+1}t_k^{q_k} - \begin{pmatrix} p_k \\ j-1 \end{pmatrix}s_k^{p_k-j+1}G_k(t_k^{q_k})\right.  \\
& \left.-\alpha_k\begin{pmatrix}   p_k\\ j-2 \end{pmatrix}
 s_k^{p_k-j+2}F_k(t_k^{q_k})\right) \otimes\cdots\otimes s_m^{p_m}t_m^{q_m}\otimes v_{(\textbf{p},\textbf{q})}.
\end{align*}
Using Lemma \ref{crucial} again, we observe that the elements
\begin{equation}\label{l}
\begin{split}
\sum_{(\textbf{p},\textbf{q})\in E}s_1^{p_1}t_1^{q_1}\otimes \cdots
\otimes &\left(\binom{p_k }{ j} s_k^{p_k-j+1}t_k^{q_k}- \binom{p_k }{ j-1}s_k^{p_k-j+1}G_k(t_k^{q_k}) \right. \\
& \qquad\quad\qquad\qquad \left.-\alpha_k \binom{p_k }{ j-2} s_k^{p_k-j+2}F_k(t_k^{q_k})\right) \otimes\cdots\otimes s_m^{p_m}t_m^{q_m}\otimes v_{(\textbf{p},\textbf{q})}
\end{split}
\end{equation}
belong to $N$ for $1\le k\le m$ and $0\le j\le P_k+2$.
Taking $j=0$ in \eqref{l}, we obtain \eqref{nnn}.
\end{proof}

\begin{lemma}\label{generator}
For any non-zero $v$ in $V$, the element $1\otimes\cdots\otimes1\otimes v$ is a generator of $\mathbf{T}$.
\end{lemma}
\begin{proof}
Let $N$ be the submodule of $\mathbf{T}$ generated by $1\otimes\cdots\otimes1\otimes v$. 
Then, from \eqref{nn} and \eqref{nnn}, we know that $\left(\bigotimes_{k=1}^{m}\O(\lambda_k,\a_k,h_k)\right)\otimes v\subset N$. 
Define $U=\left\{ u\in V\mid  (\bigotimes_{k=1}^{m}\O(\lambda_k,\a_k,h_k))\otimes u\subset N\right\}$. 
It follows that the vector space $U$ is non-zero. 
For $u\in U$, $x\in \mathcal{W}$ and $f_k\in \O(\lambda_k,\a_k,h_k)$, we have
\[
f_1\otimes\cdots\otimes f_m\otimes(xu)=x(f_1\otimes\cdots\otimes f_m\otimes u)-x(f_1\otimes\cdots\otimes f_m)\otimes u\in xN+N=N.
\]
Consequently, $U$ is a non-zero submodule of $V$, leading to the conclusion that $N=\mathbf{T}$.
\end{proof}

The \textbf{main result} of this subsection states as follows.

\begin{thm}\label{sim}
The tensor product module $\mathbf{T}$ is simple if and only if $\lambda_1,\ldots,\lambda_m$ are pairwise distinct.
\end{thm}
\begin{proof}
Suppose $\lambda_1,\ldots,\lambda_m$ are pairwise distinct and $N$ is a non-zero submodule of $\mathbf{T}$. 
Take
\[
0\ne g=\sum_{(\textbf{p},\textbf{q})\in E} \textbf{s}^{\textbf{p}}\textbf{t}^{\textbf{q}}\otimes v_{(\textbf{p},\textbf{q})}\in N
\]
with the lowest $\deg(g)$ and assume that $\deg(g)=(p'_1,\ldots,p'_m,q'_1,\ldots,q'_m)$.
\begin{claim}
We have $\deg(g)=(0,\ldots,0,0,\ldots,0)$.
\end{claim}

Indeed, if $p'_k>0$ for some $1\le k\le m$, let $k_0=\min\left\{k\mid p'_k>0 \right\}$. 
Then, by \eqref{n}, we know that the element
\[
\sum_{(\textbf{p},\textbf{q})\in E_{k_0}} \textbf{s}^{\textbf{p}-P_{k_0}e_{k_0}}\textbf{t}^{\textbf{q}}\otimes v_{(\textbf{p},\textbf{q})}
\]
belongs to $N$ and has a lower degree than $g$. 
Hence, we have $p'_k=0$ for all $1\le k\le m$ and $g=\sum_{(\textbf{0},\textbf{q})\in E}\textbf{t}^{\textbf{q}}\otimes v_{(\textbf{0},\textbf{q})}$. 
Now, by \eqref{l}, we know that for $j\in\left\{0,1,2\right\}$,
\begin{equation}\label{ll}
\sum_{(\textbf{0},\textbf{q})\in E}t_1^{q_1}\otimes \cdots
\otimes\left (\delta_{j0}t_k^{q_k}-\delta_{j1}G_k(t_k^{q_k})-\alpha_k\delta_{j2}F_k(t_k^{q_k}) \right)
\otimes\cdots\otimes t_m^{q_m}\otimes v_{(\textbf{0},\textbf{q})}\in N.
\end{equation}
Taking $j=2$ in \eqref{ll} and noting that $\a_k\ne0$, we have
\[
\sum_{(\textbf{0},\textbf{q})\in E}t_1^{q_1}\otimes \cdots
\otimes\left( g_k(t_k)t_k^{q_k}-q_kt_k^{q_k-1}\right)
\otimes\cdots\otimes t_m^{q_m}\otimes v_{(\textbf{0},\textbf{q})}\in N.
\]
Combined with \eqref{nn}, we deduce that
\[
\sum_{(\textbf{0},\textbf{q})\in E}t_1^{q_1}\otimes \cdots
\otimes q_kt_k^{q_k-1}
\otimes\cdots\otimes t_m^{q_m}\otimes v_{(\textbf{0},\textbf{q})}\in N.
\]
If $q'_k>0$ for some $1\le k\le m$, let $k_1=\min\left\{k\mid q'_k>0\right\}$. 
Then the element
\[
\sum_{(\textbf{0},\textbf{q})\in E}t_1^{q_1}\otimes \cdots
\otimes q_{k_1}t_{k_1}^{q_{k_1}-1}
\otimes\cdots\otimes t_m^{q_m}\otimes v_{(\textbf{0},\textbf{q})}
\]
belongs to $N$ and has a lower degree than $g$, contradicting the minimality.
We conclude that $g=1\otimes\cdots\otimes1\otimes v$ for some $0\ne v\in V$. 
Lemma \ref{generator} then yields $N=\mathbf{T}$.

Conversely, suppose $\lambda_i=\lambda_j$ for some $1\le i <  j\le m$. 
 It suffices to consider the tensor product module $\O(\lambda_i,\alpha_i,h_i)\bigotimes\O(\lambda_j,\alpha_j,h_j)$.
As a vector space, we have $\O(\lambda_i,\alpha_i,h_i)\bigotimes\O(\lambda_j,\alpha_j,h_j)\cong\C[s_i,s_j,t_i,t_j]$. 
Define $M=\mathrm{span}\left\{\C[t_i,t_j](s_i+s_j)^r\mid r\in\N\right\}$. 
Direct computation shows $M$ is a non-zero proper submodule of $\O(\lambda_i,\alpha_i,h_i)\bigotimes\O(\lambda_j,\alpha_j,h_j)$, though we omit the details.
\end{proof}

\subsection{Isomorphism classes}
In the previous subsection, we have established the simplicity of $\mathbf{T}$. 
The goal of this subsection is to determine their isomorphism classes.
According to Theorem \ref{sim}, throughout this subsection, we always assume $\lambda_1,\ldots,\lambda_m\in\C^*$ are pairwise distinct.

For $g\in\mathbf{T}$, define
\[
R_g=\lim_{s \to \infty} \dim\mathrm{span}\left\{g, W_ng\mid n\ge s \right\}
\]
and
\[
R_{\mathbf{T}}=\inf\left\{ R_g\mid 0\ne g\in \mathbf{T}\right\}.
\]
Let $\mathbf{N}=\bigotimes_{k=1}^{m}\C[t_k]\otimes V\subset\mathbf{T}$. Then we have the following statements.
\begin{prop}
For $0\ne g\in \mathbf{T}$, we always have $R_g\ge m+1$. 
Furthermore, the equality holds if and only if $0\ne g\in\mathbf{N}$.
\end{prop}
\begin{proof}
Assume that
\[
0\ne g=\sum_{(\textbf{p},\textbf{q})\in E} \textbf{s}^{\textbf{p}}\textbf{t}^{\textbf{q}}\otimes v_{(\textbf{p},\textbf{q})}
\]
with $\deg(g)=(p'_1,\ldots,p'_m,q'_1,\ldots,q'_m)$. Choose a sufficiently large integer $i$ such that $W_nv_{(\textbf{p},\textbf{q})}=0$ for $n\ge i$ and $(\textbf{p},\textbf{q})\in E$. For $1\le k\le m$ and $0\le j\le P_k+1$, let
\[
w_{kj}=\sum_{(\textbf{p},\textbf{q})\in E}s_1^{p_1}t_1^{q_1}\otimes \cdots\otimes \left ( \begin{pmatrix}
p_k \\
j
\end{pmatrix}s_{k}^{p_k-j}t_k^{q_k+1}+\alpha_k\begin{pmatrix}
p_k \\
j-1
\end{pmatrix}s_k^{p_k-j+1}t_k^{q_k} \right )\otimes \cdots \otimes  s_m^{p_m}t_m^{q_m}\otimes v_{(\textbf{p},\textbf{q})}.
\]
Then, by \eqref{t}, we know that for $s\ge i$,
\[
\mathrm{span}\left\{g, W_ng\mid n\ge s \right\}=\mathrm{span}\left\{g, w_{kj}\mid 1\le k\le m, 0\le j\le P_k+1\right\}.
\]
Note that $\deg(w_{10})\succ\deg(w_{20})\succ\cdots\succ\deg(w_{m0})\succ\deg(g)$.
If $p'_k>0$ for some $1\le k\le m$, let $k_0=\min\left\{k\mid p'_k>0 \right\}$, then the dimension of the subspace spanned by the elements
\begin{center}
$\left\{w_{k_0,P_{k_0}+1},g,w_{k0}\mid 1\le k\le m\right\}$
\end{center}
is $m+2$, which implies $R_g>m+1$.
If $0\ne g\in\mathbf{N}$, then we have $P_k=0$ for $1\le k\le m$. Hence  $R_g=\dim\mathrm{span}\left\{g, w_{kj}\mid 1\le k\le m, 0\le j\le1\right\}=m+1$.
\end{proof}
Take another simple tensor product module
\[
\mathbf{T}'=\bigotimes_{k=1}^{m'}\O(\lambda'_k,\a'_k,f_k)\otimes V',
\]
where $\lambda'_1,\ldots,\lambda'_{m'}\in\C^*$ are pairwise distinct and $V'$ is a simple restricted module. For $1\le k\le m'$, as a vector space, we have $\O(\lambda'_k,\a'_k,f_k)=\C[s_k,t_k]$. Set $u_k(t_k)=\dfrac{f_k(t_k)-f_k(\a'_k)}{t_k-\a'_k}$. Define the linear maps $F'_k:\C[t_k]\to\C[t_k]$ and $G'_k:\C[t_k]\to\C[t_k]$ as in \eqref{F} and \eqref{G}, respectively. Let $\mathbf{N}'=\bigotimes_{k=1}^{m'}\C[t_k]\otimes V'$.
Now, we are in a position to determine the isomorphism classes of these simple tensor product $\mathcal{W}$-modules.
\begin{thm}\label{iso}
The modules $\mathbf{T}$ and $\mathbf{T}'$ are isomorphic if and only if $m=m'$, $V\cong V'$ and $(\lambda_k,\a_k,h_k)=(\lambda'_k,\a'_k,f_k)$ up to a permutation.
\end{thm}
\begin{proof}
Since $\mathbf{T}\cong\mathbf{T}'$, we must have $R_{\mathbf{T}}=R_{\mathbf{T}'}$ and hence $m=m'$. 
Suppose $\phi:\mathbf{T}\to\mathbf{T}'$ is an isomorphism map. 
Due to $R_g=R_{\phi(g)}$ and $m=m'$, we deduce that $\phi(1\otimes\cdots\otimes1\otimes v)\in\mathbf{N}'$ for $v\in V$. 
For a fixed $0\ne v\in V$, assume that
\[
\phi(1\otimes\cdots\otimes1\otimes v)=\sum_{(\textbf{0},\textbf{r})\in S}t_1^{r_1}\otimes\cdots\otimes t_m^{r_m}\otimes v_{(\textbf{0},\textbf{r})}.
\]
Choose a sufficiently large integer $i$ such that $\mathcal{W}_nv=0$ and $\mathcal{W}_nv_{(\textbf{0},\textbf{r})}=0$ for $n\ge i$ and $(\textbf{0},\textbf{r})\in S$.
For $n\ge i$, by the action of $W_n$, we have
\begin{equation}\label{la}
\begin{split}
 &\sum_{k=1}^{m}\lambda_k^n\phi(1\otimes\cdots\otimes t_k\otimes\cdots\otimes1\otimes v)-\sum_{k=1}^{m}n\lambda_k^n\alpha_k\phi(1\otimes\cdots\otimes1\otimes v)\\
=&\sum_{k=1}^{m}(\lambda'_k)^n\sum_{(\textbf{0},\textbf{r})\in S}t_1^{r_1}\otimes\cdots\otimes t_k^{r_k+1}\otimes\cdots\otimes t_m^{r_m}\otimes v_{(\textbf{0},\textbf{r})}-\sum_{k=1}^{m}n(\lambda'_k)^n\alpha'_k\phi(1\otimes\cdots\otimes1\otimes v).
\end{split}
\end{equation}
We can renumber $\left\{\lambda_1,\ldots,\lambda_m\right\}$ if necessary. 
Note that we have $\phi(t_1\otimes1\otimes\cdots\otimes1\otimes v)\ne0$ and $\lambda_1\notin\left\{\lambda_2,\ldots,\lambda_m\right\}$.
It follows from Lemma \ref{crucial} that $\lambda_1=\lambda'_{t_1}$ for some $1\le t_1\le m$. 
Without loss of generality, we may assume $t_1=1$. 
Then equation \eqref{la} becomes
\begin{align*}
 &\lambda_1^n\phi(t_1\otimes1\otimes\cdots\otimes1\otimes v)-\lambda_1^n\sum_{(\textbf{0},\textbf{r})\in S}t_1^{r_1+1}\otimes t_2^{r_2}\otimes\cdots\otimes t_m^{r_m}\otimes v_{(\textbf{0},\textbf{r})}\\
 &-n\lambda_1^n(\alpha_1-\alpha'_1)\phi(1\otimes\cdots\otimes1\otimes v)+\sum_{k=2}^{m}\lambda_k^n\phi(1\otimes\cdots\otimes t_k\otimes\cdots\otimes1\otimes v)\\
 &-\sum_{k=2}^{m}n\lambda_k^n\alpha_k\phi(1\otimes\cdots\otimes1\otimes v)\\
=&\sum_{k=2}^{m}(\lambda'_k)^n\sum_{(\textbf{0},\textbf{r})\in S}t_1^{r_1}\otimes\cdots\otimes t_k^{r_k+1}\otimes\cdots\otimes t_m^{r_m}\otimes v_{(\textbf{0},\textbf{r})}-\sum_{k=2}^{m}n(\lambda'_k)^n\alpha'_k\phi(1\otimes\cdots\otimes1\otimes v).
\end{align*}
Note that $\lambda_1\notin\left\{\lambda_2,\ldots,\lambda_m,\lambda'_2,\ldots,\lambda'_m\right\}$ and $\phi(1\otimes t_2\otimes\cdots\otimes1\otimes v)\ne0$, 
we may assume $\lambda_2=\lambda'_2$ in a similar way. 
Repeating this process, we have $\lambda_k=\lambda'_k$ for $1\le k\le m$, and
\begin{align*}
0=&\sum_{k=1}^{m}\lambda_k^n\big(\phi(1\otimes\cdots\otimes t_k\otimes\cdots\otimes1\otimes v)-\sum_{(\textbf{0},\textbf{r})\in S}t_1^{r_1}\otimes\cdots\otimes t_k^{r_k+1}\otimes\cdots\otimes t_m^{r_m}\otimes v_{(\textbf{0},\textbf{r})}\big)\\
&-\sum_{k=1}^{m}n\lambda_k^n(\alpha_k-\alpha'_k)\phi(1\otimes\cdots\otimes1\otimes v).
\end{align*}
By Lemma \ref{crucial} and noting that $\phi(1\otimes\cdots\otimes1\otimes v)\ne0$, we know that $\alpha_k=\alpha'_k$ for $1\le k\le m$.

Assume that
\[
\phi(t_1^{q_1}\otimes\cdots\otimes t_m^{q_m}\otimes v)=\sum_{(\textbf{0},\textbf{r})\in S}t_1^{r_1+q_1}\otimes\cdots\otimes t_m^{r_m+q_m}\otimes v_{(\textbf{0},\textbf{r})}
\]
for $q_1+\cdots+q_m=t\in\N$. 
Then for $n\ge i$, by the action of $W_n$, we have
\[
\sum_{k=1}^{m}\lambda_k^n\phi(t_1^{q_1}\otimes\cdots\otimes t_k^{q_k+1}\otimes\cdots\otimes t_m^{q_m}\otimes v)=\sum_{k=1}^{m}\lambda_k^n\sum_{(\textbf{0},\textbf{r})\in S}t_1^{r_1+q_1}\otimes\cdots\otimes t_k^{q_k+r_k+1}\otimes\cdots\otimes t_m^{r_m+q_m}\otimes v_{(\textbf{0},\textbf{r})}.
\]
Using the Vandermonde determinant, we see that 
\[
\phi(t_1^{q_1}\otimes\cdots\otimes t_k^{q_k+1}\otimes\cdots\otimes t_m^{q_m}\otimes v)=\sum_{(\textbf{0},\textbf{r})\in S}t_1^{r_1+q_1}\otimes\cdots\otimes t_k^{q_k+r_k+1}\otimes\cdots\otimes t_m^{r_m+q_m}\otimes v_{(\textbf{0},\textbf{r})}
\]
hold for $1\le k\le m$.
It follows that
\begin{equation}\label{syy}
\phi(t_1^{q_1}\otimes\cdots\otimes t_m^{q_m}\otimes v)=\sum_{(\textbf{0},\textbf{r})\in S}t_1^{r_1+q_1}\otimes\cdots\otimes t_m^{r_m+q_m}\otimes v_{(\textbf{0},\textbf{r})}
\end{equation}
for $q_1,\ldots,q_m\in\N$.

For $n\ge i$, by the action of $L_n$, we have
\begin{equation}\label{yys}
\begin{split}
 &\sum_{k=1}^{m}\phi\big(1\otimes\cdots\otimes\lambda_k^n(s_k+nG_k(1)-n^2\alpha_kF_k(1))\otimes\cdots\otimes1\otimes v\big)\\
=&\sum_{k=1}^{m}\sum_{(\textbf{0},\textbf{r})\in S}t_1^{r_1}\otimes\cdots\otimes\lambda_k^n\big(s_kt_k^{r_k}+nG'_k(t_k^{r_k})-n^2\alpha_kF'_k(t_k^{r_k})\big)\otimes\cdots\otimes t_m^{r_m}\otimes v_{(\textbf{0},\textbf{r})}.
\end{split}
\end{equation}
Comparing the coefficients of $n^2\lambda_k^n$ and using Lemma \ref{crucial}, we know that
\begin{align*}
\phi\big(1\otimes\cdots\otimes F_k(1)\otimes\cdots\otimes1\otimes v\big)
=\sum_{(\textbf{0},\textbf{r})\in S}t_1^{r_1}\otimes\cdots\otimes F'_k(t_k^{r_k})\otimes\cdots\otimes t_m^{r_m}\otimes v_{(\textbf{0},\textbf{r})}.
\end{align*}
Combining this with equation \eqref{syy}, we obtain
\[
\sum_{(\textbf{0},\textbf{r})\in S}t_1^{r_1}\otimes\cdots\otimes F_k(1)t_k^{r_k}\otimes\cdots\otimes t_m^{r_m}\otimes v_{(\textbf{0},\textbf{r})}=\sum_{(\textbf{0},\textbf{r})\in S}t_1^{r_1}\otimes\cdots\otimes F'_k(t_k^{r_k})\otimes\cdots\otimes t_m^{r_m}\otimes v_{(\textbf{0},\textbf{r})},
\]
that is,
\begin{equation*}
\sum_{(\textbf{0},\textbf{r})\in S}t_1^{r_1}\otimes\cdots\otimes\big(u_k(t_k)-g_k(t_k)\big)t_k^{r_k}\otimes\cdots\otimes t_m^{r_m}\otimes v_{(\textbf{0},\textbf{r})}=\sum_{(\textbf{0},\textbf{r})\in S}t_1^{r_1}\otimes\cdots\otimes r_kt_k^{r_k-1}\otimes\cdots\otimes t_m^{r_m}\otimes v_{(\textbf{0},\textbf{r})}.
\end{equation*}
It follows that $u_k(t_k)=g_k(t_k)$ for $1\le k\le m$. Hence we have $S=\left\{(\textbf{0},\textbf{0})\right\}$ and
\[
\phi(1\otimes\cdots\otimes1\otimes v)=1\otimes\cdots\otimes1\otimes \tau(v)
\]
for some $\tau(v)\in V'$. Now equation \eqref{yys} becomes
\begin{equation}\label{yy}
\begin{split}
 & \sum_{k=1}^{m}\phi\big(1\otimes\cdots\otimes\lambda_k^n(s_k+nG_k(1)-n^2\alpha_kF_k(1))\otimes\cdots\otimes1\otimes v\big)\\
=&\sum_{k=1}^{m}1\otimes\cdots\otimes\lambda_k^n\big(s_k+nG'_k(1)-n^2\alpha_kF'_k(1)\big)\otimes\cdots\otimes1\otimes\tau(v).
\end{split}
\end{equation}
Comparing the coefficients of $n\lambda_k^n$ and using equation \eqref{syy}, we obtain $G_k(1)=G'_k(1)$. 
We have known $F_k(1)=F'_k(1)$, it follows that $h_k(\a_k)=f_k(\a_k)$ and $h_k(t_k)=f(t_k)$ for $1\le k\le m$.

Now it remains to prove $V\cong V'$. Comparing the coefficients of $\lambda_k^n$ in \eqref{yy} and using Lemma \ref{crucial}, we have
\begin{equation}\label{yyy}
\phi(1\otimes\cdots\otimes s_k\otimes\cdots\otimes1\otimes v)=1\otimes\cdots\otimes s_k\otimes\cdots\otimes1\otimes \tau(v)
\end{equation}
for all $1\le k\le m$. For any $n\in\Z$, by equations \eqref{syy} and \eqref{yyy}, we see that
\begin{align*}
 &1\otimes\cdots\otimes1\otimes \tau(L_nv)-1\otimes\cdots\otimes1\otimes L_n\tau(v)\\
=&\sum_{k=1}^{m}1\otimes\cdots\otimes\lambda_k^n\big(s_k+nG'_k(1)-n^2\alpha_k F'_k(1)\big)\otimes\cdots\otimes1\otimes\tau(v)\\
 &-\sum_{k=1}^{m}\phi\big(1\otimes\cdots\otimes\lambda_k^n(s_k+nG_k(1)-n^2\alpha_kF_k(1))\otimes\cdots\otimes1\otimes v\big)\\
=&0,
\end{align*}
which implies $\tau(L_nv)=L_n\tau(v)$. 
Similarly, using equation \eqref{syy}, we deduce that $\tau(W_nv)=W_n\tau(v)$ for any $n\in\Z$. 
Finally, it is easy to see that $\tau(Cv)=C\tau(v)$, and we conclude that the linear map $\tau: V\to V'$ is an isomorphism of $\mathcal{W}$-modules.
\end{proof}

\section{New simple modules}
Throughout this section, we will compare the simple tensor product modules $\mathbf{T}$ obtained in the previous section with other known simple $\mathcal{W}$-modules in the literature, demonstrating that these simple modules are generically new. 
The \textbf{main result} of this section is the following statement.
\begin{thm}
	The simple tensor product $\mathcal{W}$-module $\mathbf{T}$ is new except when $m=1$ and $V=\C$ is the trivial module.
\end{thm}
\begin{proof}
First, recall that a $\mathcal{W}$-module $M$ is called a weight module if $M=\bigoplus_{\beta\in\C}M_{\beta}$, where $M_{\beta}=\left\{m\in M\mid L_0m=\beta m\right\}$. 
By the action of $L_0$, for any $0\ne g\in\mathbf{T}$, we have $\deg(L_0g)\succ\deg(g)$. 
It follows that $\mathbf{T}$ is not a weight module.

The literature contains two types of non-weight simple $\mathcal{W}$-modules: 
the simple restricted modules constructed in \cite{Wang,Wang&Li,A&R,Chen&Hong&Su}, and the modules $\O(\lambda,\a)$ together with the modules $\O(\lambda,\a,h)$ with $\lambda,\a\in\C^*$. 

By the action of $W_n$, we see that
\begin{equation*}
W_n(1\otimes\cdots\otimes1\otimes v)=\sum_{k=1}^{m}\lambda_k^n(1\otimes\cdots\otimes t_k\otimes\cdots\otimes1\otimes v)+1\otimes\cdots\otimes1\otimes\left(W_nv-n\sum_{k=1}^{m}\a_k\lambda_k^nv\right)
\end{equation*}
for $0\ne v\in V$, which implies $\mathbf{T}$ is not isomorphic to any simple restricted module in \cite{Wang,Wang&Li,A&R,Chen&Hong&Su}.
Note that $W_n$ acts on $\O(\lambda,\a)$ trivially. It follows that $\mathbf{T}$ is not isomorphic to $\O(\lambda,\a)$.

Let
\begin{equation*}
	Q=\a^{-2}(W_0^2-W_{-1}W_{1})\in U(\mathcal{W}).
\end{equation*}
Then for $f(s,t)\in \O(\lambda,\a,h)$, we have $Qf(s,t)=\a^{-2}t^2f(s,t)-\a^{-2}(t+\a)(t-\a)f(s,t)=f(s,t)$, i.e. $Q$ acts on $\O(\lambda,\a,h)$ as $\mathbf{id}_{\C[s,t]}$.

Now, let us prove the module $\mathbf{T}$ is not isomorphic to any $\O(\lambda,\a,h)$ with $\lambda,\a\in\C^*$ provided that $m\ge2$.
Suppose $\varphi:\mathbf{T}\to\O(\lambda,\a,h)$ is an isomorphism of $\mathcal{W}$-modules.
Then, by the action of $Q\in U(\mathcal{W})$, we have
\begin{equation*}
\varphi(1\otimes\cdots\otimes1\otimes v)=Q\varphi(1\otimes\cdots\otimes1\otimes v)=\varphi\left(Q(1\otimes\cdots\otimes1\otimes v)\right)
\end{equation*}
for $v\in V$.
It follows that $1\otimes\cdots\otimes1\otimes v=Q(1\otimes\cdots\otimes1\otimes v)$ for every $v\in V$. 
If we identify $\mathbf{T}$ with $V[s_1,\ldots,s_m,t_1,\ldots,t_m]$, then by direct computations, 
we observe that the coefficient of $t_1t_2$ in $Q(1\otimes\cdots\otimes1\otimes v)$ is
\begin{equation*}
\a^{-2}(2-\lambda_1\lambda_2^{-1}-\lambda_2\lambda_1^{-1})v\in V,
\end{equation*}
which is non-zero since $\lambda_1\ne\lambda_2$. 
This leads to a contradiction. 

Now, we proceed to consider the case $m=1$.
If $V=\C$ is the trivial $\mathcal{W}$-module, then we certainly have $\mathbf{T}\cong \O(\lambda_1,\a_1,h_1)$.
By Theorem \ref{iso}, we see that $\mathbf{T}$ cannot be isomorphic to any $\O(\lambda,\a,h)$ if $V$ is not the trivial module.
In conclusion, the theorem is valid.
\end{proof}

\section*{Acknowledgements}
This paper is partially supported by the National Key R\&D Program of China (2024YFA1013802), 
the National Natural Science Foundation of China (11931009, 12101152, 12161141001, 12171132, 12401030 and 12401036), 
the Innovation Program for Quantum Science and Technology (2021Z D0302902), 
and the Fundamental Research Funds for the Central Universities of China.

\end{document}